\documentclass{article}

\usepackage[a4paper, 
includeheadfoot, margin = 2.54cm]{geometry}

\usepackage[T1]{fontenc}
\usepackage[latin9]{inputenc}
\usepackage{xcolor}
\usepackage[english]{babel}
\usepackage{float}
\usepackage{mathrsfs}
\usepackage{mathtools}
\usepackage{amsmath}
\usepackage{amsthm}
\usepackage{amssymb}
\usepackage[numbers]{natbib}
\usepackage[unicode=true,pdfusetitle,
 bookmarks=true,bookmarksnumbered=false,bookmarksopen=false,
 breaklinks=false,pdfborder={0 0 1},backref=false,colorlinks=false]
 {hyperref}

\makeatletter

\floatstyle{ruled}
\newfloat{algorithm}{tbp}{loa}
\providecommand{\algorithmname}{Algorithm}
\floatname{algorithm}{\protect\algorithmname}

\numberwithin{equation}{section}

\theoremstyle{plain}
\newtheorem{thm}{\protect\theoremname}
\theoremstyle{plain}
\newtheorem{lem}[thm]{\protect\lemmaname}
\theoremstyle{plain}
\newcounter{assumptioncounter}
\newtheorem{assumption}{\protect\assumptionname}

\theoremstyle{remark}
\newtheorem{rem}[thm]{\protect\remarkname}
\theoremstyle{plain}
\newtheorem{cor}[thm]{\protect\corollaryname}
\theoremstyle{plain}
\newtheorem{prop}[thm]{\protect\propositionname}

\@ifundefined{date}{}{\date{}}

\usepackage{bbm}

\usepackage{mathtools}

\allowdisplaybreaks

\usepackage{hyperref}

\usepackage[capitalise,nameinlink]{cleveref}

\AtBeginDocument{
\let\ref\cref
}

\usepackage{autonum}

\makeatletter
\newcommand{\restore@Environment}[1]{%
  \AtBeginDocument{%
    \csletcs{#1*}{#1}%
    \csletcs{end#1*}{end#1}%
  }%
}
\forcsvlist\restore@Environment{alignat,equation,gather,multline,flalign,align}
\makeatother

\usepackage{tikz}

\crefname{equation}{}{}

\crefalias{thm}{theorem}

\usepackage{etoolbox}

\ifcsmacro{defn}{}{\newtheorem{defn}{test}}
\ifcsmacro{thm}{}{\newtheorem{thm}{test}}
\ifcsmacro{lem}{}{\newtheorem{lem}{test}}
\ifcsmacro{prop}{}{\newtheorem{prop}{test}}
\ifcsmacro{cor}{}{\newtheorem{cor}{test}}
\ifcsmacro{assumption}{}{\newtheorem*{assumption*}{test}}
\ifcsmacro{example}{}{\newtheorem{example}{test}}

\let\oldthm\thm
\renewcommand{\thm}{%
\crefalias{thm}{theorem}
\oldthm
}

\let\oldlem\lem
\renewcommand{\lem}{%
\crefalias{thm}{lemma}
\oldlem
}

\let\oldprop\prop
\renewcommand{\prop}{%
\crefalias{thm}{proposition}
\oldprop
}

\let\oldcor\cor
\renewcommand{\cor}{%
\crefalias{thm}{corollary}
\oldcor
}

\let\oldexample\example
\renewcommand{\example}{%
\crefalias{thm}{example}
\oldexample
}

\let\oldassumption\assumption
\renewcommand{\assumption}{%
\crefalias{thm}{assumption}
\oldassumption
}

\let\oldrem\rem
\renewcommand{\rem}{%
\crefalias{thm}{remark}
\oldrem
}

\usepackage{crossreftools}
\let\ORGhypersetup\hypersetup
\protected\def\hypersetup{\ORGhypersetup}
\crefname{assumption}{Assumption}{Assumptions}

\makeatother

\addto\captionsamerican{\renewcommand{\assumptionname}{Assumption}}
\addto\captionsamerican{\renewcommand{\corollaryname}{Corollary}}
\addto\captionsamerican{\renewcommand{\lemmaname}{Lemma}}
\addto\captionsamerican{\renewcommand{\propositionname}{Proposition}}
\addto\captionsamerican{\renewcommand{\remarkname}{Remark}}
\addto\captionsamerican{\renewcommand{\theoremname}{Theorem}}
\addto\captionsenglish{\renewcommand{\algorithmname}{Algorithm}}
\addto\captionsenglish{\renewcommand{\assumptionname}{Assumption}}
\addto\captionsenglish{\renewcommand{\corollaryname}{Corollary}}
\addto\captionsenglish{\renewcommand{\lemmaname}{Lemma}}
\addto\captionsenglish{\renewcommand{\propositionname}{Proposition}}
\addto\captionsenglish{\renewcommand{\remarkname}{Remark}}
\addto\captionsenglish{\renewcommand{\theoremname}{Theorem}}
\providecommand{\assumptionname}{Assumption}
\providecommand{\corollaryname}{Corollary}
\providecommand{\lemmaname}{Lemma}
\providecommand{\propositionname}{Proposition}
\providecommand{\remarkname}{Remark}
\providecommand{\theoremname}{Theorem}

\global\long\def\N{\mathbb{N}}%
\global\long\def\R{\mathbb{R}}%
\global\long\def\P{\mathbb{P}}%
\global\long\def\E{\mathbb{E}}%
\global\long\def\1{\mathbbm1}%
\global\long\def\d{\mathrm{d}}%
\global\long\def\e{\mathrm{e}}%
\global\long\def\lebesgue{\lambda\mkern-13mu  \lambda}%
\global\long\def\argmin#1{\operatorname*{arg\,\min}_{#1}}%
\global\long\def\eps{\varepsilon}%
\global\long\def\theta{\vartheta}%
\global\long\def\le{\leqslant}%
\global\long\def\ge{\geqslant}%
\global\long\def\KL{\operatorname{KL}}%
\global\long\def\subset{\subseteq}%
\global\long\def\comp{\mathsf{c}}%

\foreignlanguage{english}{}
\global\long\def\tilde#1{\widetilde{#1}}%
\foreignlanguage{english}{}
\global\long\def\hat#1{\widehat{#1}}%

\begin{document}

\title{A PAC-Bayes oracle inequality for sparse neural networks}
\author{Maximilian F. Steffen and Mathias Trabs\thanks{Financial support of the DFG through project TR 1349/3-1 is gratefully
acknowledged.}}
\date{Karlsruhe Institute of Technology}
\maketitle
\begin{abstract}
\noindent We study the Gibbs posterior distribution for sparse deep neural nets in a nonparametric regression setting. The posterior can be accessed via Metropolis-adjusted Langevin algorithms. Using a mixture over uniform priors on sparse sets of network weights, we prove an oracle inequality which shows that the method adapts to
the unknown regularity and hierarchical structure of the regression function. The estimator achieves the minimax-optimal rate of convergence (up to a logarithmic factor). 
\end{abstract}

oracle inequality

\smallskip{}


\section{Introduction\label{sec:intro}}
Driven by the enormous success of neural networks in a broad spectrum
of machine learning applications, see \citet{goodfellowEtAl2016}
and \citet{schmidhuber2015} for an introduction, the theoretical
understanding of network based methods is a dynamic and flourishing
research area at the intersection of \mbox{mathematical} statistics,
optimization and approximation theory. In addition to theoretical guarantees, uncertainty
quantification is an important and challenging
problem for neural networks and has motivated the introduction of
Bayesian neural networks, where for each network weight a distribution
is learned, see \citet{graves2011} and \citet{blundellEtAl2015}
and numerous subsequent articles. In this work we study the Gibbs posterior distribution for a stochastic neural network. In a nonparametric regression problem, we show that the corresponding estimator achieves a minimax-optimal prediction risk bound up to a logarithmic factor. Moreover, the method is adaptive with respect to the unknown regularity and structure of the regression function.

\medskip{}

While early theoretical foundations for neural nets are summarized
by \citet{AnthonyBartlett1999}, the excellent approximation properties
of deep neural nets, especially with the ReLU activation function,
have been discovered in the last years, see e.g.\ \citet{yarotsky2017},
the review paper \citet{devoreEtAl2021} and \citet{bolcskeiEtAl2019}
for sparse deep neural networks. In addition to these approximation
properties, an essential explanation of the empirical capabilities
of neural networks has recently been given by \citet{schmidthieber2020}
as well as \citet{BauerKohler2019}: While classical regression methods
suffer from the curse of dimensionality, sparse neural network estimators
can profit from a hierarchical structure of the regression function
and a possibly much smaller intrinsic dimension. In particular, \citet{schmidthieber2020}
has analyzed sparse deep neural networks with ReLU activation function
which is the network class considered subsequently. To achieve a sparse neural network, a standard approach is to use penalized empirical
risk minimization which has been analyzed for neural nets by \citet{taheriEtAl2021}.

Instead of penalized empirical risk minimization, we train a neural
network by learning the Gibbs posterior distribution. The resulting stochastic neural network satisfies an oracle inequality
which implies an adaptive version of the upper bound by \citet{schmidthieber2020},
i.e., the network estimator achieves this upper bound without prior
knowledge of the hierarchical structure and the regularity of the
regression function.\medskip{}

The PAC-Bayes approach goes back to \citet{ShaweTaylor1997} and \citet{McAllester1999a,McAllester1999}.
In a quite general setting it is possible to derive upper bounds for
the prediction error either in terms of the empirical risk or in terms
of an oracle inequality. We refer to the review papers by \citet{Guedj2019}
and \citet{Alquier2021} as well as \citet{alquier2013} and \citet{steffen2023pac} for applications in sparse nonparametric regression. Empirical PAC-Bayes bounds are intensively
studied for (deep) neural nets, see \citet{Dziugaite2017,perezEtAl2021}
and further references in \citet[Section 3.3]{Alquier2021}. PAC-Bayes
oracle inequalities are less well studied. Exceptions are  \citet{cherief2020},
who has proved a (non-adaptive) oracle inequality for the variational
approximation of the PAC-Bayes procedure based on the theory by \citet{catoni2007},
and \citet{tinsiDalalyan2021} who have used the PAC-Bayes bound for
an aggregation of (shallow) neural networks. Further, \citet{franssenSzabo2022} have studied an empirical Bayes-type approach in the last layer leading to Bayesian credible sets with frequentist coverage guarantees.

\medskip{}

In order to sample from the Gibbs posterior, a Metropolis-adjusted Langevin algorithm (MALA) can be applied, see \citet{Besag1994, roberts1996b}. This apporach can exploit the well-established and efficient gradient descent algorithms. More precisely, MALA relies
on a Metropolis-Hastings algorithm where the proposal density is centered
around a gradient descent step. This yields an efficient MCMC algorithm
whose computational costs are comparable to a gradient based method.
The MCMC approach is different from the majority of literature on Bayesian
neural networks which mainly focuses on variational Bayesian inference,
see \citet{BleiEtal2017} for a review on variational inference and
\citet{graves2011} and \citet{blundellEtAl2015} for early applications
to neural networks. A main reason is that with a growing sample size the calculation of the acceptance probability in the Metropolis-Hastings step becomes expensive. To adapt to sparsity, we can choose a prior that
prefers networks with sparse weights. 

It should be noted that due to the non-convex dependence on the network
weights, the computational analysis of such an MCMC algorithm is as
challenging as the convergence analysis of gradient based methods
towards a global risk minimizer. 

\medskip{}

The paper is organized as follows: In \ref{sec:estimation}, we introduce
our estimation method and discuss its implementation. In \ref{sec:Oracle-inequality},
we state and discuss our main results for the proposed method. 
All proofs have been postponed to \ref{sec:Proofs} and some further details
on algorithmic aspects are contained in \ref{sec:Implementation}.

\section{Estimation method\label{sec:estimation}}

We consider a training sample $\mathcal{D}_{n}\coloneqq(\mathbf{X}_{i},Y_{i})_{i=1,\dots,n}\subset\R^{p}\times\R$
given by $n\in\N$ i.i.d.\ copies of generic random variables $(\mathbf{X},Y)\in\R^{p}\times\R$
on some probability space $(\Omega,\mathcal{A},\P)$ and the aim to estimate the regression function
$f\colon\R^{p}\to\R$ given by $Y=f(\mathbf{X})+\varepsilon$ with
observation error $\varepsilon$ satisfying $\E[\varepsilon\mid\mathbf{X}]=0$
almost surely (a.s.). Equivalently, $f(\mathbf{X})=\E[Y\mid\mathbf{X}]$ a.s.\ For any
estimator $\hat{f}$, the prediction risk and its empirical counterpart
are given by 
\begin{equation}
R(\hat{f})\coloneqq\E_{(\mathbf{X},Y)}\big[(Y-\hat{f}(\mathbf{X}))^{2}\big]\qquad\text{and}\qquad R_{n}(\hat{f})=\frac{1}{n}\sum_{i=1}^{n}\big(Y_{i}-\hat{f}(\mathbf{X}_{i})\big)^{2},\label{eq:R}
\end{equation}
respectively, where $\E$ denotes the expectation under $\P$ and $\E_Z$ is the (conditional) expectation only with respect to a random variable $Z$. The accuracy of the estimation procedure will be quantified
in terms of the excess risk 
\begin{equation}
\mathcal{E}(\hat{f})\coloneqq R(\hat{f})-R(f)=\E_{\mathbf{X}}\big[(\hat{f}(\mathbf{X})-f(\mathbf{X}))^{2}\big]=\Vert\hat{f}-f\Vert_{L^{2}(\P^{\mathbf{X}})}^{2}.\label{eq:excess}
\end{equation}
In the following we first introduce the considered class of sparse
neural networks. Afterwards, we introduce the estimation method and discuss its implementation. Throughout, $\vert x\vert_{q}$ denotes the $\ell^{q}$-norm
of a vector $x\in\R^{p},\,q\in[1,\infty]$. In particular, $\vert\cdot\vert\coloneqq \vert\cdot\vert_{2}$
is the Euclidean norm. For $a,b\in\R$ we write $a\lor b\coloneqq\max\{a,b\}$ and $a\land b\coloneqq \min\{a,b\}$. The cardinality of a set $\mathcal{I}$ is
denoted by $\vert\mathcal{I}\vert$. 

\subsection{Sparse neural networks}

We consider a \emph{feedforward multilayer perceptron} with $L\in\N$
hidden layers of constant width $r\in\N$. The latter restriction is purely to simplify the notation. The \emph{rectified linear unit}
(ReLU) $\sigma(x)\coloneqq x\lor 0,\,x\in\R,$ is used as activation
function. We write $\sigma_{v}x\coloneqq\big(\sigma(x_{i}+v_{i})\big)_{i=1,\dots,d}$
for vectors $x,v\in\R^{d}$. With this notation we can represent such neural networks as 
\begin{equation}
g_{\theta}(\mathbf{x})\coloneqq W^{(L+1)}\sigma_{v^{(L)}}W^{(L)}\sigma_{v^{(L-1)}}\cdots W^{(2)}\sigma_{v^{(1)}}W^{(1)}\mathbf{x}+v^{(L+1)},\qquad\mathbf{x}\in\R^{p},\label{eq:network}
\end{equation}
where the entries of the weight matrices $W^{(1)}\in\R^{r\times p},W^{(2)},\dots,W^{(L)}\in\R^{r\times r},W^{(L+1)}\in \R^{1\times r}$
and shift vectors $v^{(1)},\dots,v^{(L)}\in\R^{r},v^{(L+1)}\in\R$ are all collected in one parameter vector $\theta$. The total
number of network parameters is
\[
P\coloneqq (p+1)r+(L-1)(r+1)r+r+1.
\]
A possibly
more intuitive layer-wise representation of $g_{\theta}$ is given
by
\begin{align}
\mathbf{x}^{(0)} & \coloneqq\mathbf{x}\in\R^{p}, \\
\mathbf{x}^{(l)} & \coloneqq\sigma(W^{(l)}\mathbf{x}^{(l-1)}+v^{(l)}),\,l=1,\dots,L,\label{eq:neurons}\\
g_{\theta}(\mathbf{x})\coloneqq\mathbf{x}^{(L+1)} & \coloneqq W^{(L+1)}\mathbf{x}^{(L)}+v^{(L+1)},
\end{align}
where the activation function is applied coordinate-wise. We denote
the class of all such functions $g_{\theta}$ by $\mathcal{G}(p,L,r)$.
A network is sparse, or more precisely \emph{connection sparse}, if
many weights in the network are zero and thus some links between nodes
are inactive. For some active set $\mathcal{I}\subset\{1,\dots,P\}$,
the corresponding class of sparse networks is defined by 
\[
\mathcal{G}(p,L,r,\mathcal{I})\coloneqq\{g_{\theta}\in\mathcal{G}(p,L,r):\theta_{i}=0\text{ if }i\notin\mathcal{I}\}.
\]
For some $C\ge1$, we also introduce the class of clipped networks
\[
\mathcal{F}(p,L,r,C)\coloneqq\big\{ f_{\theta}\coloneqq(-C)\lor(g_{\theta}\land C)\,\big\vert\, g_{\theta}\in\mathcal{G}(p,L,r)\big\}
\]
and similarly we denote clipped networks with active set $\mathcal{I}$
by $\mathcal{F}(p,L,r,\mathcal{I},C)$. We abbreviate $R(\theta)\coloneqq R(f_{\theta})$
and $R_{n}(\theta)\coloneqq R_{n}(f_{\theta})$.

\subsection{Prior and posterior distribution}

In order to adopt the PAC-Bayes approach to neural networks, we first
choose a prior on the weights in the network class $\mathcal{G}(p,L,r)$.
For a given active set $\mathcal{I}$ the prior $\Pi_{\mathcal{I}}$
on the parameter set of the class $\mathcal{G}(p,L,r,\mathcal{I})$
is defined as the uniform distribution on 
\begin{equation}
\mathcal{S}_{\mathcal{I}}\coloneqq\{\theta\in[-B,B]^{P}\mid\theta_{i}=0\text{ if }i\notin\mathcal{I}\}\label{eq:SI}
\end{equation}
for some $B\ge1$. In the special case $\mathcal{I}^{\bullet}\coloneqq\{1,\dots,P\}$
we obtain a prior $\Pi^{\bullet}\coloneqq\Pi_{\mathcal{I}^{\bullet}}$
on the non-sparse network class $\mathcal{G}(p,L,r)$.  To allow
for a data-driven choice of the active set, we define the prior $\Pi$
as a mixture of the uniform priors $\Pi_{\mathcal{I}}$:
\begin{equation}
\Pi\coloneqq\sum_{i=1}^{P}2^{-i}\sum_{\substack{\mathcal{I}\subseteq\{1,\dots,P\},\\
\vert\mathcal{I}\vert=i
}
}\binom{P}{i}^{-1}\Pi_{\mathcal{I}}\Big/C_{P}\qquad\text{with}\qquad C_{P}\coloneqq(1-2^{-P}).\label{eq:prior}
\end{equation}
The basis $2$ of the geometric weights is arbitrary and can be replaced by a larger constant leading to a stronger preference of sparse networks. The theoretical results remain unchanged up to constants.
The prior $\Pi$ can be understood as a hierarchical prior, where
we first draw a geometrically distributed sparsity $i$, given $i$
we uniformly choose an active set $\mathcal{I}\subset\{1,\dots,P\}$
with $\vert \mathcal{I}\vert =i$ and on $\mathcal{I}$ the uniform prior $\Pi_{\mathcal{I}}$
is applied.

Based on \emph{$\Pi$,} the \emph{Gibbs posterior} probability distribution
$\Pi_{\lambda}(\cdot\mid\mathcal{D}_{n})$ is defined via 
\begin{equation}
\Pi_{\lambda}(\mathrm{d}\theta\mid\mathcal{D}_{n})\propto\exp(-\lambda R_{n}(\theta))\Pi(\mathrm{d}\theta)\label{eq:posterior}
\end{equation}
with the so-called \emph{inverse temperature parameter} $\lambda>0$ and empirical
prediction risk from \ref{eq:R}. While \ref{eq:posterior} coincides
with the classical Bayesian posterior distribution if $Y_{i}=f_{\theta}(\mathbf{X}_{i})+\eps_{i}$
with i.i.d.\ $\eps_{i}\sim\mathcal{N}(0,n/(2\lambda))$, the so-called
tempered likelihood $\exp(-\lambda R_{n}(\theta))$ serves as a proxy
for the unknown distribution 
\emergencystretch 3em 
of the observations given $\theta$. As we will see, the method is indeed applicable under quite general
assumptions on the regression model. The Gibbs posterior distribution
weights each $\theta$ based on its empirical performance on the data,
where the inverse temperature parameter $\lambda$ determines the impact of $R_{n}(\theta)$
in comparison to the prior beliefs.

An estimator for $f$ can be obtained by drawing from the posterior
distribution, i.e.,
\begin{equation}
\widehat{f}_{\lambda}\coloneqq f_{\widehat{\theta}_{\lambda}}\qquad\text{for}\qquad\widehat{\theta}_{\lambda}\mid\mathcal{D}_n\sim\Pi_{\lambda}(\cdot\mid\mathcal{D}_{n})\label{eq:Estimator}
\end{equation}
or in the form of the posterior mean
\begin{equation}
\bar{f}_{\lambda}\coloneqq\E\big[f_{\hat{\theta}_{\lambda}}\,\big\vert\,\mathcal{D}_{n}\big]=\int f_{\theta}\,\Pi_{\lambda}(\d\theta\mid\mathcal{D}_{n}).\label{eq:postMean}
\end{equation}
The maximum a posteriori (MAP)
estimator could be used as well, but we will focus on the previous two estimators. 

\subsection{MCMC algorithm}\label{sec:MCMC}

Since the normalizing constant of the posterior distribution is unknown,
the posterior distribution itself is not accessible in practice. Popular
Bayesian networks rely on a variational Bayes approach, cf. \citet{BleiEtal2017},
and approximate the posterior by some easier distribution. For instance,
the \emph{Bayes by backprop} method by \citet{blundellEtAl2015} uses
independent normal distributions for the weights such that the training
of the Bayesian network reduces to the calibration of the means and
variances for all weights. In contrast, the classical approach in
Bayesian statistics is to sample $\widehat{\theta}_{\lambda}$ by
constructing a Markov-chain $(\theta^{(k)})_{k\in\N_{0}}$ with stationary
distribution $\Pi_{\lambda}(\cdot\mid\mathcal{D}_{n})$ with a Markov
chain Monte Carlo (MCMC) method, see \citet{RobertCasella2004}. 
MCMC methods aim for the exact posterior distribution and have been successfully
applied to neural networks, see e.g.\ \citet{Zhang2020cyclicSGMCMC}, \citet{ssgMCMC2022}.

For ease of presentation, we will discuss MALA for the prior $\Pi^{\bullet}$ on the non-sparse network
class $\mathcal{G}(p,L,r)$, i.e.\ $\Pi^{\bullet}$ is the uniform
distribution on $[-B,B]^{P}$. The corresponding posterior distribution
is denoted by $\Pi_{\lambda}^{\bullet}(\mathrm{d}\theta\mid\mathcal{D}_{n})\propto\exp(-\lambda R_{n}(\theta))\Pi^{\bullet}(\mathrm{d}\theta)$. 

Applying the generic Metropolis-Hastings algorithm to $\Pi_{\lambda}^{\bullet}(\mathrm{d}\theta\mid \mathcal{D}_{n})$
and taking into account that the prior $\Pi^{\bullet}$ is uniform,
we obtain the following iterative method: Starting with some initial
choice $\theta^{(0)}\in\R^{P}$, we successively generate
$\theta^{(k+1)}$ given $\theta^{(k)}$, $k\in\N_{0}$, by 

\[
\theta^{(k+1)}=\begin{cases}
\tau^{(k)} & \text{with probability }\alpha(\tau^{(k)}\mid\theta^{(k)})\\
\theta^{(k)} & \text{with probability }1-\alpha(\tau^{(k)}\mid\theta^{(k)})
\end{cases},
\]
where $\tau^{(k)}$ is a random variable drawn from some conditional
proposal density $q(\cdot\mid\theta^{(k)})$ and the \emph{acceptance
probability} is given by

\[
\alpha(\tau^{(k)}\mid\theta^{(k)})=\min\Big\{1,\exp\big(-\lambda R_{n}(\tau^{(k)})+\lambda R_{n}(\theta^{(k)})\big)\1_{[-B,B]^{P}}(\tau^{(k)})\frac{q(\theta^{(k)}\mid\tau^{(k)})}{q(\tau^{(k)}\mid\theta^{(k)})}\Big\}.
\]
The Metropolis-Hastings acceptance step ensures that $(\theta^{(k)})_{k\in\N_{0}}$ is a Markov
chain with invariant distribution $\Pi^{\bullet}(\cdot\mid\mathcal{D}_{n})$.
The convergence to the invariant distribution follows from \citet[Theorem 2.2]{Roberts1996}.

The practical success of the Metropolis-Hastings algorithm fundamentally depends on the
choice of the proposal distribution. In the construction of $q(\tau\mid \theta)$,
MALA exploits the well-known gradient approach for an empirical risk
minimizer 
\[
\theta^{\ast}\in\argmin{\theta}R_{n}(\theta)
\]
with the empirical risk $R_{n}(\theta)=R_{n}(f_{\theta})$ from \ref{eq:R}.
The gradient $\nabla_{\theta}R_{n}(\theta)$ of $R_{n}(\theta)$ with
respect to $\theta$ can be computed efficiently using backpropagation.
The gradient descent method commonly used to train neural networks
would suggest to tweak a given $\theta^{(k)}$ in the direction of
the gradient, that is
\[
\tilde{\theta}^{(k+1)}=\theta^{(k)}-\gamma\nabla_{\theta}R_{n}(\theta^{(k)})
\]
with a \emph{learning rate} $\gamma>0$. Therefore, MALA proposes a normal
distribution around $\tilde{\theta}^{(k+1)}$ with some standard deviation
$s>0$. This standard deviation should not be too large as otherwise
the acceptance probability might be too small. As a result the proposal
would rarely be accepted, the chain might not be sufficiently randomized
and the convergence to the invariant target distribution would be
too slow in practice. On the other hand, $s$ should not be smaller
than the shift $\gamma\nabla_{\theta}R_{n}(\theta^{(k)})$ in the
mean, since otherwise $q(\theta^{(k)}\mid\tau^{(k)})$ might be too
small. Therefore, the probability density $q$ of the proposal distribution
is given by the density
\begin{equation}
q(\tau\mid\theta)=\frac{1}{(2\pi s^{2})^{P/2}}\exp\Big(-\frac{1}{2s^{2}}\big\vert\tau-\theta+\gamma\nabla_{\theta}R_{n}(\theta)\big\vert^{2}\Big).\label{eq:proposal}
\end{equation}

To calculate the estimators $\hat{f}_{\lambda}$ and $\bar{f}_{\lambda}$
from \ref{eq:Estimator} and \ref{eq:postMean}, respectively, one
chooses a \emph{burn-in} time $b\in\N$ to let the distribution of
the Markov chain stabilize at its invariant distribution and then
sets
\[
\hat{f}_{\lambda}\coloneqq f_{\theta^{(b)}}\qquad\text{and}\qquad\bar{f}_{\lambda}\coloneqq\frac{1}{N}\sum_{k=1}^{N}f_{\theta^{(b+ck)}}.
\]
A sufficiently large \emph{gap length} $c\in\N$ ensures the necessary
variability and an approximate independence between $\theta^{(b+ck)}$
and $\theta^{(b+c(k+1))}$, whereas $N\in\N$ has to be large enough
for a good approximation of the expectation by the empirical mean.

Note that the gradient has to be calculated only once in each MCMC
iteration. While we use the standard gradient, the calculation of
$\nabla_{\theta}R_{n}$ could be realized and combined with state-of-art
gradient based methods such as mini-batch stochastic gradient descent
or Adam \citep{kingmaBa2014}. Compared to the training of a non-stochastic network, the additional computational price of MALA is determined by the computation of the acceptance probability $\alpha(\tau^{(k)}\mid\theta^{(k)})$ and a larger number of necessary iterations due to
the potential rejection of proposals.

In order to extend the Metropolis-Hastings algorithm from the full
prior $\Pi^{\bullet}$ to the mixing prior $\Pi$ from \ref{eq:prior},
we have to take into account the hierarchical structure of $\Pi$.
Hence, we use the reversible-jump Markov chain Monte Carlo (RJMCMC)
algorithm first proposed by \citet{green1995}. In the context of
a PAC-Bayes method it has been discussed in \citet{Guedj2019}. In
particular, the RJMCMC algorithm has been successfully applied by
\citet{alquier2013} in a high-dimensional regression setting. Some
further details on the  of neural networks with
the RJMCMC algorithm are given in \ref{sec:Implementation}.

\section{Oracle inequality\label{sec:Oracle-inequality}}

In this section we state the theoretical guarantees for the Gibbs posterior distribution. 
As a benchmark for the performance of the method, we define the \emph{oracle choice}
on $\mathcal{S}_{\mathcal{I}}$ from \ref{eq:SI} for some active set $\mathcal{I}$ as 
\begin{equation}
\theta_{\mathcal{I}}^{\ast}\in\argmin{\theta\in\mathcal{S}_{\mathcal{I}}}R(\theta).\label{eq:oracle}
\end{equation}
The oracle is not accessible to the practitioner because $R(\theta)$ depends
on the unknown distribution of $(\mathbf{X},Y)$. A solution to the
minimization problem in \ref{eq:oracle} always exists since $\mathcal{S}_{\mathcal{I}}$
is compact and $\theta\mapsto R(f_{\theta})$ is continuous. If there
is more than one solution, we choose one of them. Our main result
gives a theoretical guarantee that the PAC-Bayes estimator $\widehat{f}_{\lambda}$
from \ref{eq:Estimator} is at least as good as the oracle $\theta_{\mathcal{I}}^{\ast}$
in terms of the excess risk. To this end, we need some mild assumptions
on the regression model.
\begin{assumption}
\label{assu:bounded}~
\begin{enumerate}
\item For $K,C\ge1$ we have $\E[\vert\mathbf{X}\vert^{2}]\le pK$
and $\Vert f\Vert_{\infty}\le C$.
\item $\eps$ is conditionally on $\mathbf{X}$ sub-Gaussian. More precisely,
there are constants $\sigma,\Gamma>0$ such that 
\[
\E[\vert\varepsilon\vert^{k}\mid\mathbf{X}]\le\frac{k!}{2}\sigma^{2}\Gamma^{k-2}\text{a.s.},\qquad\text{for all }k\ge2.
\]
\end{enumerate}
\end{assumption}

Note that neither the loss function nor the data are assumed to be
bounded. We obtain the following non-asymptotic oracle inequality:
\begin{thm}[PAC-Bayes oracle inequality]\label{thm:oracleinequality} 
Under \ref{assu:bounded} there are constants $\Xi_0,\Xi_1>0$ only depending on $C,\Gamma,\sigma$ such that for $\lambda=n/\Xi_0$ and sufficiently large $n$, we have with probability of at least $1-\delta$, that
\begin{equation}
\mathcal{E}(\hat{f}_{\lambda})\le\min_{\mathcal{I}}\Big(4\mathcal{E}(f_{\theta_{\mathcal{I}}^{\ast}})+\frac{\Xi_{1}}{n}\big(\vert\mathcal{I}\vert L\log(p\lor n)+\log(2/\delta)\big)\Big).\label{eq:oracleIneq}
\end{equation}
\end{thm}
\begin{rem}
$\Xi_0$ can explicitly be chosen as $\Xi_0=16(C^2+\sigma^2)+ 16C(\Gamma\lor(2C))$, the dependence of $\Xi_1$ on $C,\Gamma,\sigma$ is at most quadratic and $n\ge n_0\coloneqq 2\lor r\lor b\lor K$ is sufficiently large.
\end{rem}
The right-hand side of \ref{eq:oracleIneq} can be interpreted similarly
to the classical bias-variance decomposition in nonparametric statistics.
The first term $\mathcal{E}(f_{\theta_{\mathcal{I}}^{\ast}})=\E[(f_{\theta_{\mathcal{I}}^{\ast}}(\mathbf{X})-f(\mathbf{X}))^{2}]$
quantifies the approximation error while second term is an upper bound
for the stochastic error. In particular, we recover $\frac{\vert\mathcal{I}\vert}{n}\log p$
as the typical error term for estimating high-dimensional vectors
with sparsity $\vert\mathcal{I}\vert$. The factor $4$ in the upper
bound can be improved to $(1+\eta)$ for any $\eta>0$ at the cost
of a larger constant $\Xi_{1}$. 

\ref{thm:oracleinequality} is in line with classical PAC-Bayes oracle
inequalities, see \citet{Alquier2021}. \citet{cherief2020} has obtained
a similar oracle inequality for a variational approximation of the
Gibbs posterior distribution, but without the minimum over the active
sets. For penalized empirical risk minimization, \citet{taheriEtAl2021}
have obtained another oracle inequality with a different dependence
on the depth $L$.

The $1-\delta$ probability takes into account the randomness of the data and of the estimate.
Denoting
\begin{equation}
r_{n,p}^{2}\coloneqq\min_{\mathcal{I}}\Big(4\Vert f_{\theta_{\mathcal{I}}^{*}}-f\Vert_{L^{2}(\P^{\mathbf{X}})}^{2}+\frac{\Xi_{1}}{n}\vert\mathcal{I}\vert L\log(p\lor n)\Big),\label{eq:r2}
\end{equation}
we can rewrite \ref{eq:oracleIneq} as
\[
\E\big[\Pi_{\lambda}\big(\Vert f_{\hat{\theta}_{\lambda}}-f\Vert_{L^{2}(\P^{\mathbf{X}})}^{2}>r_{n,p}^{2}+t^{2}\,\big\vert\,\mathcal{D}_{n}\big)\big]\le2\e^{-nt^{2}/\Xi_{1}},\qquad t>0,
\]
which is a \emph{contraction rate} result in terms of a frequentist
analysis of the nonparametric Bayes method.

An immediate consequence of the previous theorem is an analogous oracle
inequality of the posterior mean $\bar{f}_{\lambda}$ from \ref{eq:postMean}.
\begin{cor}[Posterior mean]
\label{cor:mean}Under the conditions of \ref{thm:oracleinequality}
we have with probability of at least $1-\delta$ that
\[
\mathcal{E}(\bar{f}_{\lambda})\le\min_{\mathcal{I}}\Big(4\mathcal{E}(f_{\theta_{\mathcal{I}}^{\ast}})+\frac{\Xi_{2}}{n}\big(\vert\mathcal{I}\vert L\log(p\lor n)+\log(2/\delta)\big)\Big)
\]
with a constant $\Xi_{2}$ only depending on $C,\Gamma,\sigma$.
\end{cor}

The infimum over all $\mathcal{I}$ in the oracle inequalities shows
that the estimator may choose a sparse neural network for approximating
$f$ in a data-driven way. In particular, $\hat{f}_{\lambda}$ as
well as $\bar{f}_{\lambda}$ adapt to the unknown regularity and a
potentially lower dimensional structure of $f$. Using the approximation
properties of neural networks, the oracle inequality yields the optimal
rate of convergence (up to a logarithmic factor) over the following
class of hierarchical functions:
\begin{align*}
\mathcal{H}(q,\mathbf{d},\mathbf{t},\beta,C_{0}) & \coloneqq\Big\{ g_{q}\circ\dots\circ g_{0}\colon[0,1]^{p}\to\R\,\Big\vert\,g_{i}=(g_{ij})_{j}^{\top}\colon[a_{i},b_{i}]^{d_{i}}\to[a_{i+1},b_{i+1}]^{d_{i+1}},\\
 & \qquad\qquad\qquad\qquad\qquad\qquad\qquad g_{ij}\text{ depends on at most \ensuremath{t_{i}} arguments,}\\
 & \qquad\qquad\qquad\qquad\qquad\qquad\qquad g_{ij}\in\mathcal{C}_{t_{i}}^{\beta_{i}}([a_{i},b_{i}]^{t_{i}},C_{0}),\text{ for some }\vert a_{i}\vert,\vert b_{i}\vert\le C_{0}\Big\},
\end{align*}
where $\mathbf{d}\coloneqq(p,d_{1},\dots,d_{q},1)\in\N^{q+2},\mathbf{t}\coloneqq(t_{0},\dots,t_{q})\in\N^{q+1},\beta\coloneqq(\beta_{0},...,\beta_{q})\in(0,\infty)^{q+1}$
and where $\mathcal{C}_{t_{i}}^{\beta_{i}}([a_{i},b_{i}]^{t_{i}},C_{0})$
denote classical H\"older balls with H\"older regularity $\beta_{i}>0$ and radius $C_0\ge 1$.
\ref{thm:oracleinequality} reveals the following adaptive version
of the upper bound by \citet{schmidthieber2020}: 
\begin{prop}[Rates of convergence]
\label{prop:rate}
Let $\log p\le n/(\log^2 n)$ and $\mathbf{X}\in[0,1]^{p}$. In the situation of \ref{thm:oracleinequality}, there exists a network
architecture $(L,r)=(C_{1}\lceil\log_{2}n\rceil,C_2n)$
with $C_{1}$ and $C_{2}$ only depending on upper bounds for $q,\vert(d_{1},\dots,d_{q})\vert_{\infty},\vert\mathbf{t}\vert_{\infty},\vert\beta\vert_{\infty}$
and $C_{0}$ such that the estimators $\hat{f}_{\lambda}$ and $\bar{f}_{\lambda}$
yield an excess risk for sufficiently large n uniformly over all hierarchical
functions $f\in\mathcal{H}(q,\mathbf{d},\mathbf{t},\beta,C_{0})$
of at most 
\begin{align*}
\mathcal{E}(\hat{f}_{\lambda}) & \le\Xi_{3}\Big(\frac{\log(p\lor n)\log^{2}(n)}{n}\Big)^{2\beta^{\ast}/(2\beta^{*}+t^{\ast})}+\Xi_{3}\frac{\log(2/\delta)}{n}\qquad\text{and}\\
\mathcal{E}(\bar{f}_{\lambda}) & \le\Xi_{4}\Big(\frac{\log(p\lor n)\log^{2}(n)}{n}\Big)^{2\beta^{\ast}/(2\beta^{*}+t^{\ast})}+\Xi_{4}\frac{\log(2/\delta)}{n}
\end{align*}
with probability of at least $1-\delta$, respectively, where
$\beta^{*}$ and $t^{*}$ are given by 
\[
\beta^\ast\coloneqq\beta^\ast_{i^\ast},\qquad t^\ast\coloneqq t^\ast_{i^\ast}\qquad\text{for}\qquad i^\ast\in \argmin{i=0,\dots,q}\frac{2\beta_{i}^{*}}{2\beta_{i}^{*}+t^{\ast}_{i}}\qquad\text{and}\qquad 
\beta_{i}^{\ast}\coloneqq\beta_{i}\prod_{l=i+1}^{q}(\beta_{l}\land 1).
\]
The constants $\Xi_{3}$ and $\Xi_{4}$ only depend on $q,(d_{1},\dots,d_{q}),\mathbf{t},\beta$
and $C_{0}$ and $C,\Gamma,\sigma$.
\end{prop}

For a fixed dimension $p$ it has been proved by \citet{schmidthieber2020}
that this rate is indeed optimal in a minimax sense up to a logarithmic
factor in $n$. Studying classical H\"older balls $\mathcal{C}_{p}^{\beta}([0,1]^{p},C_{0})$,
a contraction rate of order $n^{-2\beta/(2\beta+p)}$ has been derived by \citet{polsonRockova2018} and \citet{cherief2020}, where \citet{polsonRockova2018}
have used a hierarchical prior to obtain adaptivity with respect to
$\beta$ and the result by \citet{cherief2020} is non-adaptive, but
incorporates a variational Bayes method.

\section{Proofs\label{sec:Proofs}}
We first provide some preliminary results in \ref{sec:PACbound} before we prove the oracle inequality in \ref{sec:ProofOracle}. All remaining proofs can be found in the last two subsections.
\subsection{A PAC-Bayes bound}\label{sec:PACbound}

Let $\mu,\nu$ be probability measures on a measurable space $(E,\mathscr{A})$.
The \emph{Kullback-Leibler divergence} of $\mu$ with respect to $\nu$
is defined via 
\[
\KL(\mu\mid\nu)\coloneqq\begin{cases}
\int\log\big(\frac{\d\mu}{\d\nu}\big)\,\d\mu, & \text{if }\mu\ll\nu\\
\infty, & \text{otherwise}
\end{cases}.
\]
The following classical lemma is a key ingredient for PAC-Bayes bounds.
A proof can be found in \citet[p. 159]{catoni2004} or \citet{Alquier2021}.
\begin{lem}
\label{lem:classicallemma}Let $h\colon E\to\R$ be a measurable function
such that $\int\exp\circ h\,\d\mu<\infty$. With the convention $\infty-\infty=-\infty$,
it then holds that 
\begin{equation}
\log\Big(\int\exp\circ h\,\d\mu\Big)=\sup_{\nu}\Big(\int h\,\d\nu-\KL(\nu\mid\mu)\Big),\label{eq:gibbsequality}
\end{equation}
where the supremum is taken over all probability
measures $\nu$ on $(E,\mathscr{A})$. If additionally, $h$ is bounded
from above on the support of $\mu$, then the supremum in \ref{eq:gibbsequality}
is attained for $\nu=g$ with the Gibbs distribution $g$, i.e.\ $\frac{\d g}{\d \mu}:\propto\exp\circ h$.
\end{lem}

Note that no generality is lost by considering only those probability
measures $\nu$ on $(E,\mathscr{A})$ such that $\nu\ll\mu$ and thus
\[
\log\Big(\int\exp\circ h\,\d\mu\Big)=-\inf_{\nu\ll\mu}\Big(\KL(\nu\mid\mu)-\int h\,\d\nu\Big).
\]

Next we state a Bernstein type inequality for the excess risk $\mathcal{E}(\theta)\coloneqq\mathcal{E}(f_{\theta})=R(\theta)-\E_{(\mathbf{X},Y)}[(Y-f(\mathbf{X}))^{2}]$.
Its proof is given in \ref{sec:RemainingProofs}.
\begin{lem}
\label{lem:Bernstein}Set $\mathcal{E}(\theta)\coloneqq R(\theta)-R(f)$
and $\mathcal{E}_{n}(\theta)\coloneqq R_{n}(\theta)-R_{n}(f)$. For
$C_{n,\lambda}\coloneqq\frac{\lambda}{n}\frac{8(C^{2}+\sigma^{2})}{1-w\lambda/n}$
and $w\coloneqq16C(\Gamma\lor2C)$, we have for all $\lambda\in[0,n/w)$
that
\[
\max\big\{\E\big[\exp\big(\lambda\big(\mathcal{E}(\theta)-\mathcal{E}_{n}(\theta)\big)\big)\big],\E\big[\exp\big(\lambda\big(\mathcal{E}_{n}(\theta)-\mathcal{E}(\theta)\big)\big)\big]\big\}\le\exp\big(C_{n,\lambda}\lambda\mathcal{E}(\theta)\big).
\]
\end{lem}

Based on these two lemmas, we can verify a PAC-Bayes bound for the
excess risk. The basic proof strategy is standard in the PAC-Bayes
literature, see e.g.\ \citet{alquier2013}.
\begin{prop}[PAC-Bayes bound]
\label{prop:mainpaclemma} For any $\mathcal{D}_{n}$-dependent (in
a measurable way) probability measure $\rho\ll\Pi$ and any $\lambda\in(0,n/w)$
such that $C_{n,\lambda}\le1/2$, we have
\begin{equation}
\mathcal{E}(\widehat{\theta}_{\lambda})\le3\int\mathcal{E}\,\d\rho+\frac{4}{\lambda}\big(\KL(\rho\mid\Pi)+\log(2/\delta)\big)\label{eq:mainpacbound}
\end{equation}
with probability
of at least $1-\delta$. 
\end{prop}

\begin{proof}
\ref{lem:Bernstein} yields
\[
\max\big\{\E\big[\exp\big(\lambda(1-C_{n,\lambda})\mathcal{E}(\theta)-\lambda\mathcal{E}_{n}(\theta)-\log\delta^{-1}\big)\big],\E\big[\exp\big(\lambda\mathcal{E}_{n}(\theta)-\lambda(1+C_{n,\lambda})\mathcal{E}(\theta)-\log\delta^{-1}\big)\big]\big\}\le\delta.
\]
Integrating in $\theta$ with respect to the prior probability measure
$\Pi$ and applying Fubini's theorem, we conclude
\begin{align}
\E\Big[\int\exp\big(\lambda(1-C_{n,\lambda})\mathcal{E}(\theta)-\lambda\mathcal{E}_{n}(\theta)-\log\delta^{-1}\big)\,\d\Pi(\theta)\Big] & \le\delta\qquad\text{and}\nonumber \\
\E\Big[\int\exp\big(\lambda\mathcal{E}_{n}(\theta)-\lambda(1+C_{n,\lambda})\mathcal{E}(\theta)-\log\delta^{-1}\big)\,\d\Pi(\theta)\Big] & \le\delta.\label{eq:BersteinConsequence}
\end{align}
For the posterior distribution $\Pi_{\lambda}(\cdot\mid\mathcal{D}_{n})\ll\Pi$
with corresponding Radon-Nikodym density 
\begin{equation}
\frac{\d\Pi_{\lambda}(\theta\mid\mathcal{D}_{n})}{\d\Pi}=D_{\lambda}^{-1}\exp(-\lambda R_{n}(\theta)),\qquad D_{\lambda}\coloneqq\int\exp(-\lambda R_{n}(\theta))\,\d\Pi(\theta)\label{eq:posteriorDensity}
\end{equation}
 with respect to $\Pi$, we obtain
\begin{align*}
\E_{\mathcal{D}_{n},\hat{\theta}\sim\Pi_{\lambda}(\cdot\mid\mathcal{D}_{n})}\big[ & \exp\big(\lambda(1-C_{n,\lambda})\mathcal{E}(\hat{\theta})-\lambda\mathcal{E}_{n}(\hat{\theta})-\log\delta^{-1}+\lambda R_{n}(\hat{\theta})+\log D_{\lambda}\big)\big]\\
= & \;\,\E_{\mathcal{D}_{n},\hat{\theta}\sim\Pi_{\lambda}(\cdot\mid \mathcal{D}_{n})}\Big[\exp\Big(\lambda(1-C_{n,\lambda})\mathcal{E}(\hat{\theta})-\lambda\mathcal{E}_{n}(\hat{\theta})-\log\delta^{-1}-\log\Big(\frac{\d\Pi_{\lambda}(\hat{\theta}\mid\mathcal{D}_{n})}{\d\Pi}\Big)\Big)\Big]\\
= & \;\,\E_{\mathcal{D}_{n}}\Big[\int\exp\big(\lambda(1-C_{n,\lambda})\mathcal{E}(\theta)-\lambda\mathcal{E}_{n}(\theta)-\log\delta^{-1}\big)\,\d\Pi(\theta)\Big]\le\delta.
\end{align*}
Since $\1_{[0,\infty)}(x)\le \e^{\lambda x}$ for all $x\in\R$, we
deduce for $\hat{\theta}\sim\Pi_\lambda(\cdot\mid\mathcal{D}_n)$ with probability not larger than $\delta$ that
\[
\big(1-C_{n,\lambda}\big)\mathcal{E}(\hat{\theta})-\mathcal{E}_{n}(\hat{\theta})+R_{n}(\hat{\theta})-\lambda^{-1}\big(\log\delta^{-1}-\log D_{\lambda}\big)\ge0.
\]
Provided $\big(1-C_{n,\lambda}\big)>0$, we thus have with probability
of at least $1-\delta$:
\[
\mathcal{E}(\hat{\theta})\le\frac{1}{1-C_{n,\lambda}}\big(-R_{n}(f)+\lambda^{-1}\big(\log\delta^{-1}-\log D_{\lambda}\big)\big).
\]
\ref{lem:classicallemma} yields
\begin{equation}
-\log D_{\lambda}=-\log\Big(\int\exp(-\lambda R_{n}(\theta))\,\d\Pi(\theta)\Big)=\inf_{\rho\ll\Pi}\Big(\KL(\rho\mid\Pi)+\int\lambda R_{n}(\theta)\,\d\rho(\theta)\Big).\label{eq:constant}
\end{equation}
Therefore, we have with  probability
of at least $1-\delta$:
\[
\mathcal{E}(\hat{\theta})\le\frac{1}{1-C_{n,\lambda}}\inf_{\rho\ll\Pi}\Big(\int\mathcal{E}_{n}(\theta)\,\d\rho(\theta)+\lambda^{-1}\big(\log\delta^{-1}+\KL(\rho\mid\Pi)\big)\Big).
\]
In order to reduce the integral $\int\mathcal{E}_{n}(\theta)\,\d\rho(\theta)$
to $\int\mathcal{E}(\theta)\,\d\rho(\theta)$, we use Jensen's inequality
and \ref{eq:BersteinConsequence} to obtain for any probability measure
$\rho\ll\Pi$ (which may depend on $\mathcal{D}_{n}$)
\begin{align*}
\E_{\mathcal{D}_{n}}\Big[\exp\Big( & \int\lambda\mathcal{E}_{n}(\theta)-\lambda(1+C_{n,\lambda})\mathcal{E}(\theta)\,\d\rho(\theta)-\KL(\rho\mid\Pi)-\log\delta^{-1}\Big)\Big]\\
= & \;\,\E_{\mathcal{D}_{n}}\Big[\exp\Big(\int\lambda\mathcal{E}_{n}(\theta)-\lambda(1+C_{n,\lambda})\mathcal{E}(\theta)-\log\Big(\frac{\d\rho}{\d\Pi}(\theta)\Big)-\log\delta^{-1}\,\d\rho(\theta)\Big)\Big]\\
\le & \:\,\E_{\mathcal{D}_{n},\theta\sim\rho}\Big[\exp\Big(\lambda\mathcal{E}_{n}(\theta)-\lambda(1+C_{n,\lambda})\mathcal{E}(\theta)-\log\Big(\frac{\d\rho}{\d\Pi}(\theta)\Big)-\log\delta^{-1}\Big)\Big]\\
= & \:\,\E_{\mathcal{D}_{n}}\Big[\int\exp\big(\lambda\mathcal{E}_{n}(\theta)-\lambda(1+C_{n,\lambda})\mathcal{E}(\theta)-\log\delta^{-1}\big)\,\d\Pi(\theta)\Big]\le\delta.
\end{align*}
Using $\1_{[0,\infty)}(x)\le \e^{\lambda x}$ again, we conclude with
probability of at least $1-\delta$:
\[
\int\mathcal{E}_{n}(\theta)\,\d\rho(\theta)\le(1+C_{n,\lambda})\int\mathcal{E}(\theta)\,\d\rho(\theta)+\lambda^{-1}\big(\KL(\rho\mid\Pi)+\log\delta^{-1}\big).
\]
Therefore, we conclude with probability of at least $1-2\delta$
\[
\mathcal{E}(\hat{\theta})\le\frac{1}{1-C_{n,\lambda}}\inf_{\rho\ll\Pi}\Big((1+C_{n,\lambda})\int\mathcal{E}(\theta)\,\d\rho(\theta)+\frac{2}{\lambda}\big(\log\delta^{-1}+\KL(\rho\mid\Pi)\big)\Big)
\]
which yields the claimed bound if $C_{n,\lambda}\le1/2.$
\end{proof}

\subsection{Proof of \ref{thm:oracleinequality}}\label{sec:ProofOracle}

We fix some index set $\mathcal{I}$, a radius $\eta\in(0,1]$ and
apply \ref{prop:mainpaclemma} with $\rho=\rho_{\mathcal{I},\eta}$
defined via
\[
\frac{\d\rho_{\mathcal{I},\eta}}{\d\Pi_{\mathcal{I}}}(\theta)\propto\1_{\{\vert\theta-\theta_{\mathcal{I}}^{\ast}\vert_{\infty}\le\eta\}}
\]
with $\theta_{\mathcal{I}}^{*}$ from \ref{eq:oracle}. Note that
indeed $C_{n,\lambda}\le1/2$ for the given choice of $\lambda$.
In order to control the integral term, we decompose
\begin{align}
\int\mathcal{E}\,\mathrm{d}\rho & =\mathcal{E}(\theta_{\mathcal{I}}^{\ast})+\int\E\big[(f_{\theta}(\mathbf{X})-f(\mathbf{X}))^{2}-(f_{\theta_{\mathcal{I}}^{*}}(\mathbf{X})-f(\mathbf{X}))^{2}\big]\,\d\rho(\theta)\nonumber \\
 & =\mathcal{E}(\theta_{\mathcal{I}}^{\ast})+\int\E[(f_{\theta_{\mathcal{I}}^{\ast}}(\mathbf{X})-f_{\theta}(\mathbf{X}))^{2}]\,\mathrm{d}\rho(\theta)+2\int\E[(f(\mathbf{X})-f_{\theta_{\mathcal{I}}^{\ast}}(\mathbf{X}))(f_{\theta_{\mathcal{I}}^{\ast}}(\mathbf{X})-f_{\theta}(\mathbf{X}))]\,\d\rho(\theta)\nonumber \\
 & \le\mathcal{E}(\theta_{\mathcal{I}}^{\ast})+\int\E[(f_{\theta_{\mathcal{I}}^{\ast}}(\mathbf{X})-f_{\theta}(\mathbf{X}))^{2}]\,\mathrm{d}\rho(\theta)\nonumber \\
 & \qquad+2\int\E\big[(f(\mathbf{X})-f_{\theta_{\mathcal{I}}^{\ast}}(\mathbf{X}))^{2}\big]^{1/2}\E\big[(f_{\theta_{\mathcal{I}}^{\ast}}(\mathbf{X})-f_{\theta}(\mathbf{X}))^{2}\big]^{1/2}\,\mathrm{d}\rho(\theta)\\
 & \le\frac{4}{3}\mathcal{E}(\theta_{\mathcal{I}}^{\ast})+4\int\E[(f_{\theta_{\mathcal{I}}^{\ast}}(\mathbf{X})-f_{\theta}(\mathbf{X}))^{2}]\,\mathrm{d}\rho(\theta),\label{eq:intErho}
\end{align}
using $2ab\le\frac{a^{2}}{3}+3b^{2}$ in the last step. To bound the
remainder, we use the Lipschitz continuity of the map $\theta\mapsto f_{\theta}(\mathbf{\mathbf{x}})$
for fixed $\mathbf{x}$:
\begin{lem}
\label{lem:Lipschitz}Let $\theta,\tilde{\theta}\in[-B,B]^{ P}$.
Then we have for $\mathbf{x}\in\R^{p}$ that
\[
\vert f_{\theta}(\mathbf{x})-f_{\tilde{\theta}}(\mathbf{x})\vert\le4(2rB)^{L}(\vert\mathbf{x}\vert_{1}\lor1)\cdot\vert\theta-\tilde{\theta}\vert_{\infty}.
\]
\end{lem}

Due to the support of $\rho=\rho_{\mathcal{I},\eta}$, we obtain 

\begin{equation}
\int\mathcal{E}\, \d\rho\le\frac{4}{3}\mathcal{E}(\theta_{\mathcal{I}}^{\ast})+\frac{4}{n^{2}}\qquad\text{for}\qquad\eta=\frac{1}{8(2rB)^LpK\,n}.\label{eq:integralterm}
\end{equation}
It remains to bound the Kullback-Leibler term in \ref{eq:mainpacbound}
which can be done with the following lemma:
\begin{lem}
\label{lem:aux_a} We have
\begin{enumerate}
\item[(i)] ${\displaystyle \KL(\rho_{\mathcal{I},\eta}\mid\Pi)=\KL(\rho_{\mathcal{I},\eta}\mid \Pi_{\mathcal{I}})+\log(C_{\mathcal{I}})\quad\text{where}\quad C_{\mathcal{I}}\coloneqq C_{P}2^{\vert\mathcal{I}\vert}\binom{P}{\vert\mathcal{I}\vert},}$
\item[(ii)] ${\displaystyle \KL(\rho_{\mathcal{I},\eta}\mid\Pi_{\mathcal{I}})\le\vert\mathcal{I}\vert\log(2B/\eta).}$
\end{enumerate}
In particular, 
\[
\KL(\rho_{\mathcal{I},\eta}\mid\Pi)\le\vert\mathcal{I}\vert\log(2B/\eta )+\log(C_{\mathcal{I}}).
\]
\end{lem}

Together with ${P \choose \vert\mathcal{I}\vert}\le\frac{P^{\vert \mathcal{I}\vert }}{(\vert \mathcal{I}\vert)!}\le(\e P)^{\vert\mathcal{I}\vert}$
by Stirling's formula, the previous lemma yields
\begin{align}
\KL(\rho\mid \Pi) & \le\vert\mathcal{I}\vert\log\big(4 C_{P}BP\e/\eta\big).\label{eq:KLterm}
\end{align}
Plugging \ref{eq:integralterm} and \ref{eq:KLterm} into the PAC-Bayes
bound \ref{eq:mainpacbound}, we conclude
\begin{align}
\mathcal{E}(\widehat{\theta}_{\lambda}) & \le4\mathcal{E}(\theta_{\mathcal{I}}^{\ast})+\frac{4}{n^{2}}+\frac{4}{\lambda}\big(\vert\mathcal{I}\vert\log\big(32BP\e(2rB)^{L}pKn\big)+\log(2/\delta)\big). \\
 & \le4\mathcal{E}(\theta_{\mathcal{I}}^{\ast})+\frac{\Xi_1}{n}\big(\vert\mathcal{I}\vert L\log(p\lor n)+\log(2/\delta)\big)\label{eq:preoracle}
\end{align}
for $n\ge n_0\coloneqq 2 \lor r\lor B\lor K$ and some constant $\Xi_1$ only depending on $C,\Gamma,\sigma$. Note
that the upper bound in \ref{eq:preoracle} is deterministic and $\mathcal{I}$
is arbitrary. Therefore, we can choose $\mathcal{I}$ such that this
bound is minimized, which completes the proof.\hfill\qed

\subsection{Remaining proofs for \ref{sec:Oracle-inequality}}

\subsubsection{Proof of \ref{cor:mean}}

Jensen's and Markov's inequality yield for $r_{n,p}^{2}$ from \ref{eq:r2}
that
\begin{align*}
\P\Big(\mathcal{E}(\bar{f}_{\lambda})>r_{n,p}^{2}+\frac{\Xi_{1}}{n}+\frac{\Xi_{1}}{n}\log(2/\delta)\Big) & =\P\Big(\Vert\E [f_{\hat{\theta}_{\lambda}}\mid \mathcal{D}_{n}]-f\Vert_{L^{2}(\P^{\mathbf{X}})}^{2}>r_{n,p}^{2}+\frac{\Xi_{1}}{n}+\frac{\Xi_{1}}{n}\log(2/\delta)\Big)\\
 & \le\P\Big(\E\big[\Vert f_{\hat{\theta}_{\lambda}}-f\Vert_{L^{2}(\P^{\mathbf{X}})}^{2}\,\big\vert\,\mathcal{D}_{n}\big]>r_{n,p}^{2}+\frac{\Xi_{1}}{n}+\frac{\Xi_{1}}{n}\log(2/\delta)\Big)\\
 & =\P\Big(\int_{\frac{\Xi_{1}}{n}\log(2/\delta)}^{\infty}\Pi_{\lambda}\big(\Vert f_{\hat{\theta}_{\lambda}}-f\Vert_{L^{2}(\P^{\mathbf{X}})}^{2}>r_{n,p}^{2}+t\,\big\vert\,\mathcal{D}_{n}\big)\,\d t>\frac{\Xi_{1}}{n}\Big)\\
 & \le\frac{n}{\Xi_{1}}\int_{\frac{\Xi_{1}}{n}\log(2/\delta)}^{\infty}\E\big[\Pi_{\lambda}\big(\Vert f_{\hat{\theta}_{\lambda}}-f\Vert_{L^{2}(\P^{\mathbf{X}})}^{2}>r_{n,p}^{2}+t\,\big\vert\,\mathcal{D}_{n}\big)\big]\,\d t.
\end{align*}
Using \ref{thm:oracleinequality}, we conclude
\[
\P\Big(\mathcal{E}(\bar{f}_{\lambda})>r_{n,p}^{2}+\frac{\Xi_{1}}{n}+\frac{\Xi_{1}}{n}\log(2/\delta)\Big)\le\frac{2n}{\Xi_{1}}\int_{\frac{\Xi_{1}}{n}\log(2/\delta)}^{\infty}\e^{-nt/\Xi_{1}}\,\d t=\delta.\tag*{{\qed}}
\]

\subsubsection{Proof of \ref{prop:rate}}

Throughout, denote by $C_{i},i=1,2,\dots$ constants only depending
on upper bounds for $q$, $\vert(d_{1},\dots,d_{q})\vert_{\infty}$,
$\vert\mathbf{t}\vert_{\infty}$, $\vert\beta\vert_{\infty}$ and
$C_{0}.$ 

We will verify that for any sufficiently large $n,M\in\N$ there exists
a sparse ReLU neural network $g=g_{\theta}\in\mathcal{G}(p,C_{1}\lceil\log_{2}n\rceil,C_{2}M,\mathcal{I})$
with $\vert \mathcal{I}\vert \le C_{3}M\lceil\log_{2}n\rceil$ and $\vert\theta\vert_{\infty}\le1\le B$
such that
\begin{equation}
\Vert g-f\Vert_{L^{\infty}([0,1]^p)}\le C_{4}M^{-\beta^{\ast}/t^{\ast}}.\label{eq:approxerror}
\end{equation}

Careful inspection of the proof of \citet[Theorem 1]{schmidthieber2020}
reveals that there exists a sparse ReLU neural network $g\in\mathcal{G}(p,L,r,\mathcal{J})$
with weights and shifts absolutely bounded by $1$ and 
\begin{align*}
L & =3(q-1)+\sum_{i=0}^{q}\big(8+(\lceil\log_{2}n\rceil+5)(1+\lceil\log_{2}(t_{i}\lor\beta_{i})\rceil)\big),\\
r & =6M\max_{i=0,\dots,q}d_{i+1}(t_{i}+\lceil\beta_{i}\rceil)\qquad\text{and}\\
\vert\mathcal{J}\vert & \le  \sum_{i=0}^{q}d_{i+1}\big(141(t_{i}+\beta_{i}+1)^{3+t_{i}}M(\lceil\log_{2}n\rceil+6)+4\big)
\end{align*}
such that \ref{eq:approxerror} holds with 
\[
M=\Big\lceil\Big(\frac{n}{\log^{2}(n)\log(p\lor n)}\Big)^{t^{\ast}/(2\beta^{\ast}+t^{\ast})}\Big\rceil\le n,\label{eq:chooseNbalance}
\]
provided $M\ge\max_{i=0,\dots,q}(\beta_{i}+1)^{t_{i}}\lor(C_{0}(2C_0)^{\beta_i}+1)\e^{t_{i}}.$
Hence, it remains to show that $g$ can also be represented as a ReLU
neural network in 
\begin{equation}
\mathcal{G}(p,C_{1}\lceil\log_{2}n\rceil,C_{2}n,\mathcal{I}).\label{eq:enlargednetwork}
\end{equation}
To do this, we employ the \emph{embedding properties of network function
classes }from \citet[Section 7.1]{schmidthieber2020}.

Note that $L$, $r$ and the upper bound for $\vert\mathcal{J}\vert$ are
independent of $d_{0}=p$ and monotonically increasing in $q$, $\vert (d_1,\dots,d_q)\vert_\infty$, $\vert \mathbf{t}\vert_\infty$ and $\vert \beta\vert_\infty$. Also, $r$ is of order $M$, $L$ is of order $\lceil\log_2 n\rceil$ and the
upper bound for $\vert\mathcal{J}\vert$ is of order $M\lceil\log_2 n\rceil$.
Using the \emph{enlarging} and the \emph{depth synchronisation properties},
$g$ can indeed be written as a ReLU neural network in \ref{eq:enlargednetwork}.
Note that to ensure the depth of the network, we added additional
layers after the last hidden layer, instead of right after the input
to preserve the order of the sparsity.

\ref{thm:oracleinequality} together with  $\mathcal{E}(f_{\theta_{\mathcal{I}}^{\ast}})\le\Vert g_{\theta}-f\Vert_{L^{\infty}([0,1]^{p})}^{2}$ now yields
\begin{equation}
    \mathcal{E}(\hat{\theta}_\lambda)\le 4C_4 M^{-2\beta^\ast/t^\ast} +\frac{\Xi_1C_3}{n}M\lceil\log_2 n\rceil^2\log(p\lor n)+\Xi_1\frac{\log(2/\delta)}{n}
\end{equation}
with probability of at least $1-\delta$.

The convergence rate for the posterior mean is obtained analogously using \ref{cor:mean}.\hfill\qed
\subsection{Proofs of the auxiliary results\label{sec:RemainingProofs}}

\subsubsection{Proof of \ref{lem:Bernstein}}

We write $\mathcal{E}_{n}(\theta)=\frac{1}{n}\sum_{i=1}^{n}Z_{i}$
with centered and independent random variables
\begin{align*}
Z_{i} & \coloneqq(Y_{i}-f_{\theta}(\mathbf{X}_{i}))^{2}-(Y_{i}-f(\mathbf{X}_{i}))^{2}=-\big(2\eps_{i}+f(\mathbf{X}_{i})-f_{\theta}(\mathbf{X}_{i})\big)\big(f_{\theta}(\mathbf{X}_{i})-f(\mathbf{X}_{i})\big)
\end{align*}
Since $f$ and $f_{\theta}$ are bounded by $C$ and $\eps_{i}$ is
sub-Gaussian we have
\[
\E\big[Z_{i}^{2}\big]=\E\big[\big(2\eps_{i}+f(\mathbf{X}_{i})-f_{\theta}(\mathbf{X}_{i})\big)^{2}\big(f_{\theta}(\mathbf{X}_{i})-f(\mathbf{X}_{i})\big)^{2}\big]\le2(4\sigma^{2}+4C^{2})\mathcal{E}(\theta)\eqqcolon U
\]
and for $k\ge3$
\begin{align*}
\E[(Z_{i})_{+}^{k}] & \le\E\big[\vert2\eps_{i}+f(\mathbf{X}_{i})-f_{\theta}(\mathbf{X}_{i})\vert^{k}\vert f_{\theta}(\mathbf{X}_{i})-f(\mathbf{X}_{i})\vert^{k-2}(f_{\theta}(\mathbf{X}_{i})-f(\mathbf{X}_{i}))^{2}\big]\\
 & \le(2C)^{k-2}\E\big[\vert2\eps_{i}+f(\mathbf{X}_{i})-f_{\theta}(\mathbf{X}_{i})\vert^{k}(f_{\theta}(\mathbf{X})-f(\mathbf{X}))^{2}\big]\\
 & \le(2C)^{k-2}2^{k-1}\big(k!2^{k-1}\sigma^{2}\Gamma^{k-2}+(2C)^{k}\big)\mathcal{E}(\theta)\\
 & \le(2C)^{k-2}k!8^{k-2}\big(\Gamma^{k-2}\vee(2C)^{k-2}\big)v=k!Uw^{k-2}.
\end{align*}
In view of $\E[\mathcal{E}_{n}(\theta)]=\mathcal{E}(\theta)$,
Bernstein's inequality \citep[inequality (2.21)]{Massart2007} yields
\[
\E\big[\exp\big(\lambda\big(\mathcal{E}_{n}(\theta)-\mathcal{E}(\theta)\big)\big)\big]\le\exp\Big(\frac{U\lambda^{2}}{n(1-w\lambda/n)}\Big).
\]
The same bound remains true if we replace $Z_{i}$ by $-Z_{i}$.\hfill\qed

\subsubsection{Proof of \ref{lem:Lipschitz}}
Set $\eta\coloneqq\vert\theta-\tilde{\theta}\vert_{\infty}$ and let $W^{(1)},\dots,W^{(L+1)},v^{(1)},\dots,v^{(L+1)}$
and $\tilde W^{(1)},\dots,\tilde W^{(L+1)},\tilde v^{(1)},\dots,\tilde v^{(L+1)}$
be the weights and shifts associated with $\theta$ and $\tilde{\theta}$,
respectively. Define $\tilde{\mathbf{x}}^{(l)}$, $l=0,\dots,L+1$, analogously
to \ref{eq:neurons}. We can recursively deduce from the Lipschitz-continuity
of $\sigma$ that for $l=2,\dots,L$: 
\begin{align}
\vert\mathbf{x}^{(1)}\vert_{1} & \le\vert W^{(1)}\mathbf{x}\vert_{1}+\vert v^{(1)}\vert_{1}\\
 & \le2rB(\vert\mathbf{x}\vert_{1}\lor1),\\
\vert\mathbf{x}^{(1)}-\tilde{\mathbf{x}}^{(1)}\vert_{1} & \le\vert W^{(1)}\mathbf{x}^{(0)}+v^{(1)}-\tilde{W}^{(1)}\tilde{\mathbf{x}}^{(0)}-\tilde{v}^{(1)}\vert_{1}\\
 & \le\eta2r(\vert\mathbf{x}\vert_{1}\lor1),\\
\vert\mathbf{x}^{(l)}\vert_{1} & \le\vert W^{(l)}\mathbf{x}^{(l-1)}\vert_{1}+\vert v^{(l)}\vert_{1}\\
 & \le2rB(\vert\mathbf{x}^{(l-1)}\vert_{1}\vee1)\qquad\text{and}\\
\vert\mathbf{x}^{(l)}-\tilde{\mathbf{x}}^{(l)}\vert_{1} & \le\vert W^{(l)}\mathbf{x}^{(l-1)}+v^{(l)}-\tilde{W}^{(l)}\tilde{\mathbf{x}}^{(l-1)}-\tilde{v}^{(l)}\vert_{1}\\
 & \le\vert(W^{(l)}-\tilde{W}^{(l)})\mathbf{x}^{(l-1)}\vert_{1}+\vert\tilde{W}^{(l)}(\mathbf{x}^{(l-1)}-\tilde{\mathbf{x}}^{(l-1)})\vert_{1}+\vert v^{(l)}-\tilde{v}^{(l)}\vert_{1}\\
 & \le\eta2r(\vert\mathbf{x}^{(l-1)}\vert_{1}\lor1)+rB\vert\mathbf{x}^{(l-1)}-\tilde{\mathbf{x}}^{(l-1)}\vert_{1}.
\end{align}
Therefore, 
\begin{align}
\vert\mathbf{x}^{(L)}\vert_{1} & \le(2rB)^{L-1}(\vert\mathbf{x}^{(1)}\vert_{1}\lor1)\\
 & \le(2rB)^{L}(\vert\mathbf{x}\vert_{1}\lor1)\qquad\text{and}\\
\vert\mathbf{x}^{(L)}-\tilde{\mathbf{x}}^{(L)}\vert_{1} & \le\eta2r\sum_{k=1}^{L-1}(rB)^{k-1}(\vert\mathbf{x}^{(L-k)}\vert_{1}\lor1)+(rB)^{L-1}\vert\mathbf{x}^{(1)}-\tilde{\mathbf{x}}^{(1)}\vert_{1}\\
 & \le\eta2^{(L+1)}r(\vert\mathbf{x}\vert_{1}\lor1)(rB)^{L-1}
\end{align}
Since the clipping function $y\mapsto(-C)\vee(y\wedge C)$ has Lipschitz
constant $1$, we conclude 
\begin{align}
\vert f_{\theta}(\mathbf{x})-f_{\tilde{\theta}}(\mathbf{x})\vert & \le\vert g_{\theta}(\mathbf{x})-g_{\tilde{\theta}}(\mathbf{x})\vert\\
 & =\vert\mathbf{\mathbf{x}}^{(L+1)}-\tilde{\mathbf{x}}^{(L+1)}\vert\\
 & =\vert W^{(L+1)}\mathbf{x}^{(L)}+v^{(L+1)}-\tilde{W}^{(L+1)}\tilde{\mathbf{x}}^{(L)}-\tilde{v}^{(L+1)}\vert\\
 & \le\vert(W^{(L+1)}-\tilde{W}^{(L+1)})\mathbf{x}^{(L)}\vert+\vert\tilde{W}^{(L+1)}(\mathbf{x}^{(L)}-\tilde{\mathbf{x}}^{(L)})\vert+\vert v^{(L+1)}-\tilde{v}^{(L+1)}\vert\\
 & \le r\vert W^{(L+1)}-\tilde{W}^{(L+1)}\vert_{\infty}\vert\mathbf{x}^{(L)}\vert_{1}+r\vert\tilde{W}^{(L+1)}\vert_{\infty}\vert\mathbf{x}^{(L)}-\tilde{\mathbf{x}}^{(L)}\vert_{1}+\vert v^{(L+1)}-\tilde{v}^{(L+1)}\vert\\
 & \le\eta r(2rB)^{L}(\vert \mathbf{x}\vert_{1}\lor1)+\eta(rB)^{L}2^{L+1}(\vert\mathbf{x}\vert_{1}\lor1)+\eta\\
 & \le\eta4(2rB)^{L}(\vert\mathbf{x}\vert_{1}\lor1).\tag*{{\qed}}
\end{align}

\subsubsection{Proof of \ref{lem:aux_a}}

\emph{(i)} We will show that
\begin{equation}
\frac{\d\rho_{\mathcal{I},\eta}}{\d\Pi}=C_{\mathcal{I}}\frac{\d\rho_{\mathcal{I},\eta}}{\d\Pi_{\mathcal{I}}}.\label{eq:rhodensity}
\end{equation}
from which we can deduce
\[
\KL(\rho_{\mathcal{I},\eta}\mid\Pi)=\int\log\Big(\frac{\d\rho_{\mathcal{I},\eta}}{\d\Pi}\Big)\,\d\rho_{\mathcal{I},\eta}=\int\log\Big(\frac{\d\rho_{\mathcal{I},\eta}}{\d\Pi_{\mathcal{I}}}\Big)\,\d\rho_{\mathcal{I},\eta}+\log(C_{\mathcal{I}})=\KL(\rho_{\mathcal{I}, \eta}\mid \Pi_{\mathcal{I}})+\log(C_{\mathcal{I}}).
\]
For \ref{eq:rhodensity}, we need to show
\begin{equation}
\rho_{\mathcal{I},\eta}(A)=\int_{A}C_{\mathcal{I}}^{-1}\frac{\d\rho_{\mathcal{I},\eta}}{\d\Pi}\,\d\Pi_{\mathcal{I}}\label{eq:rhodensity2}
\end{equation}
for all $A\in\mathscr{B}_{\R^{P}}$. Observe that for the sets
\[
\mathcal{S}_{\mathcal{J},\Leftrightarrow}\coloneqq\{\theta\in\mathcal{S_{\mathcal{J}}}\mid\theta_{i}\ne0\Leftrightarrow i\in\mathcal{J}\}
\]
with $\emptyset\ne\mathcal{J}\subseteq\{1,\dots,P\}$, we have
\begin{equation}
\Pi_{\mathcal{J}}(\mathcal{S}_{\mathcal{J},\Leftrightarrow})=1.\label{eq:eqneq}
\end{equation}
In particular, \ref{eq:eqneq} holds for $\mathcal{J}=\mathcal{I}$.
Since also $\rho_{\mathcal{I},\eta}(\mathcal{S}_{\mathcal{I},\Leftrightarrow})=1$,
no generality is lost in additionally assuming $A\subseteq\mathcal{S}_{\mathcal{I},\Leftrightarrow}.$
Note how
\begin{equation}
\mathcal{S}_{\mathcal{J},\Leftrightarrow}\cap\mathcal{S}_{\mathcal{I},\Leftrightarrow}=\emptyset\qquad\forall\mathcal{J}\ne\mathcal{I}.\label{eq:emptysets}
\end{equation}
Combining \ref{eq:eqneq} with \ref{eq:emptysets}, we see that
\[
\int_{A}\frac{\d\rho_{\mathcal{I},\eta}}{\d\Pi}\,\d\Pi_{\mathcal{J}}=0\qquad\forall\mathcal{J}\ne\mathcal{I}.
\]
Therefore, 
\[
\rho_{\mathcal{I},\eta}(A)=\int_{A}\frac{\d\rho_{\mathcal{I},\eta}}{\d\Pi}\,\d\Pi=\int_{A}C_{\mathcal{I}}^{-1}\frac{\d\rho_{\mathcal{I},\eta}}{\d\Pi}\,\d\Pi_{\mathcal{I}}.
\]

\bigskip

\noindent\emph{(ii)} Since $\rho_{\mathcal{I},\eta}$ and $\Pi_{\mathcal{I}}$
are product measures, their KL-divergence is equal to the sum of the
KL-divergences in each of the $P$ factors. For factors with index
$i\notin\mathcal{I}$, this is zero, as both factors have all their
mass in $0$. For factors with index $i\in\mathcal{I}$, we are comparing
\[
\mathcal{U}([(\theta_{\mathcal{I}}^{\ast})_{i}-\eta,(\theta_{\mathcal{I}}^{\ast})_{i}+\eta]\cap[-B,B])\qquad\text{with}\qquad\mathcal{U}([-B,B]).
\]
The KL-divergence of these distributions is equal to
\[
\log\Big(\frac{\lebesgue([-B,B])}{\lebesgue([(\theta_{\mathcal{I}}^{\ast})_{i}-\eta,(\theta_{\mathcal{I}}^{\ast})_{i}+\eta]\cap[-B,B])}\Big)\le\log\Big(\frac{\lebesgue([-B,B])}{\lebesgue([0,\eta])}\Big)=\log(2B/\eta).
\]
Thus,
\[
\KL(\rho_{\mathcal{I},\eta}\mid \Pi_{\mathcal{I}})=\sum_{i\in\mathcal{I}}\KL\big(\mathcal{U}([(\theta_{\mathcal{I}}^{\ast})_{i}-\eta,(\theta_{\mathcal{I}}^{\ast})_{i}+\eta]\cap[-B,B])\mid\mathcal{U}([-B,B])\big)\le\vert\mathcal{I}\vert\log(2B/\eta).\tag*{{\qed}}
\]

\section{Implementation of the mixing prior distribution\label{sec:Implementation}}

To incorporate sparse networks, the algorithm from \ref{sec:MCMC} has to be extended. The following augmentations to the proposal density are inspired by \citet{alquier2013}.

For a fixed active set $\mathcal{I}$ we modify the proposal density
from \ref{eq:proposal} to 
\[
\psi_{\mathcal{I}}\big((\tau)_{i\in\mathcal{I}}\mid\theta\big)=\frac{1}{(2\pi s^{2})^{\vert \mathcal{I}\vert /2}}\exp\Big(-\frac{1}{2 s^{2}}\sum_{i\in\mathcal{I}}\big(\tau_{i}-\theta_{i}+\gamma\tfrac{\partial}{\partial\theta_{i}}R_{n}(\theta)\big)^{2}\Big)
\]
and set $\tau_{i}=0$ for all $i\notin\mathcal{I}$. The mixture over
sparse active sets leads to the conditional proposal density
\begin{align*}
q(\tau & \mid\theta)=\frac{q_{-}(\tau\mid\theta)+2q_{=}(\tau\mid\theta)+q_{+}(\tau\mid\theta)}{4}\1_{\{1<\vert\mathcal{I}\vert<P\}}\\
 & \qquad\qquad+\frac{q_{+}(\tau\mid\theta)+2q_{=}(\tau\mid\theta)}{3}\1_{\{\vert\mathcal{I}\vert=1\}}+\frac{q_{-}(\tau\mid\theta)+2q_{=}(\tau\mid\theta)}{3}\1_{\{\vert\mathcal{I}\vert=P\}},
\end{align*}
with $\mathcal{I}=\{i:\theta_{i}\ne0\}$ and some $q_{-}$ and $q_{+}$
proposing to remove or add a component, respectively, while $q_{=}$
leaves the active set untouched. Specifically, we choose 
\begin{gather*}
q_{=}(\tau\mid\theta)=\psi_{\mathcal{I}}(\tau\mid\theta),\quad q_{-}(\tau\mid\theta)=\sum_{i\in\mathcal{I}}w_{i}^{-}\psi_{\mathcal{I}\setminus\{i\}}(\tau\mid\theta)\quad\text{and}\quad q_{+}(\tau\mid\theta)=\sum_{i\in\mathcal{I}^{\comp}}w_{i}^{+}\psi_{\mathcal{I}\cup\{i\}}(\tau\mid\theta)
\end{gather*}
with
\[
w_{i}^{-}=\exp(-\vert\theta_{i}\vert)\Big/\sum_{j\in\mathcal{I}}\exp(-\vert\theta_{j}\vert)\qquad\text{and}\qquad w_{i}^{+}=\tilde{w}_{i}^{+}\Big/\sum_{j\in\mathcal{I}^{\comp}}\tilde{w}_{j}^{+},
\]
where $\tilde{w}_{i}^{+}=\vert\{j\in\mathcal{I}^{\comp}:\vert\frac{\partial}{\partial\theta_{j}}R_{n}(\theta)\vert\le\vert\frac{\partial}{\partial\theta_{i}}R_{n}(\theta)\vert\}\vert^{2}$.
The weights $\tilde{w}_{i}^{-}$ are chosen such that (absolutely)
smaller entries have a larger probability of being removed. On the
other hand, $\tilde{w}_{i}^{+}$ are chosen such that components with
a large impact (as measured by $\nabla_{\theta}R_{n}(\theta)$) on
the model have a higher probability of being included. Note that this
extension of the algorithm requires no additonal computation of $\nabla_{\theta}R_{n}(\theta)$,
as this was already done in the previous MCMC step. Finally, the mixing
prior from \ref{eq:prior} also has to be accounted for in the acceptance
probabilities, leading to
\[
\alpha(\tau\mid\theta)=\min\Big\{1,\exp\big(-\lambda R_{n}(\tau)+\lambda R_{n}(\theta)\big)\frac{\Pi(\tau)q(\theta\mid\tau)}{\Pi(\theta)q(\tau\mid\theta)}\Big\},
\]
where with some abuse of notation $\Pi(\theta)$ denotes the probability
density of the prior. 

\bibliographystyle{apalike2}
\bibliography{references}

\end{document}